\documentclass{amsart}
\usepackage{amssymb}
\usepackage{graphicx}
\usepackage{hyperref}

\newtheorem{theorem}{Theorem}[section]
\newtheorem{lemma}[theorem]{Lemma}

\theoremstyle{definition}
\newtheorem{definition}[theorem]{Definition}

\theoremstyle{remark}

\numberwithin{equation}{section}

\newcommand{\N}{\ensuremath{\mathbb{N}}}
\newcommand{\R}{\ensuremath{\mathbb{R}}}

\begin{document}

\title[Perfect colourings of simplices and hypercubes]{Perfect colourings 
of simplices and hypercubes in dimension four and five with few colours}

\author{Dirk Frettl\"oh}
\address{Bielefeld University, Postfach 100131, 33501 Bielefeld, Germany}
\email{dirk.frettloeh@udo.edu}
\urladdr{https://math.uni-bielefeld.de/\textasciitilde frettloe}


\date{}

\begin{abstract}
A vertex colouring of some graph is called \emph{perfect} if each vertex 
of colour $i$ has the same number $a_{ij}$ of neighbours of colour $j$. 
Here we determine all perfect colourings of the edge graphs of the hypercube
in dimensions 4 and 5 by two and three colours, respectively.
For comparison we list all perfect colourings of the edge graphs of 
the simplex in dimensions 4 and 5, respectively. 
\end{abstract}

\maketitle

\section{Introduction}

Perfect colourings of graphs are colourings with the following property:
If some vertex of colour $i$ has exactly $a_{i1}$ neighbours of colour 1, 
exactly $a_{i2}$ neighbours of colour 2 and so on, then \emph{all} vertices 
of colour $i$ have exactly $a_{i1}$ neighbours of colour 1, exactly $a_{i2}$ 
neighbours of colour 2 and so on. Figure \ref{fig:perfbsp} shows a perfect 
colouring  of the edge graph of the cube with three colours. The matrix 
$(a_{ij})_{ij}$ is the \emph{colour adjacency} of the perfect colouring.

Perfect colourings appear throughout the literature under several different 
names: equitable partitions, completely regular vertex sets, distance partitions,
association schemes, etc; and in in several contexts: algebraic graph theory, 
combinatorial designs, coding theory, finite geometry. For instance, each
distance partition of a distance regular graph is a perfect colouring, but not 
vice versa. Similarly, any subgroup   of the automorphism group of a graph $G$ 
induces a perfect colouring of $G$ by considering the vertex orbits of the subgroup 
\cite[Sec.~9.3]{GR}. However, not every perfect colouring arises from a graph 
automorphism. For a broader overview see \cite{FI,I}. 

Perfect 2-colourings of hypercube graphs were already studied in \cite{F},
with emphasis on existence in arbitrary dimension.
Some concrete perfect colourings for small graphs were constructed 
for instance in \cite{AM,AAb,AvMo,GG,M}. In \cite{DF} we generalized several
results from those papers. In particular we determined all colour adjacency
matrices of perfect colourings with 2, 3, or 4 colours for all 3-, 4-, and
5-regular graphs, and list all perfect colourings of the edge graph of the 
Platonic solids with 2, 3, or 4 colours. In the present paper we apply and 
extend the results of \cite{DF} to find all perfect colourings of the edge graphs
of the simplices in dimension four and five, and  all perfect 2-colourings 
and 3-colourings of the edge graphs of the hypercubes in dimension 
four and five. For the 
sake of brevity, let us denote the edge graph of the simplex in $\R^d$ 
by \emph{$d$-simplex}, and the edge graph of the hypercube in $\R^d$
by \emph{$d$-cube}.
\begin{figure}
\[ \includegraphics[width=50mm]{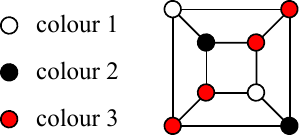} \]
\caption{A perfect colouring of the edge graph of the cube with three colours.
The corresponding colour adjacency matrix is $\Big( \begin{smallmatrix}
0 & 1 & 2\\ 1 & 0 & 2 \\ 1 & 1 & 1 \end{smallmatrix} \Big)$.  
\label{fig:perfbsp}
}
\end{figure}
\section{Preliminaries}
Throughout the paper let $G=(V,E)$ be a finite, undirected, simple, loop-free graph.
A partition of $V$ into disjoint nonempty sets $V_1, \ldots, V_m$ is called 
an \emph{$m$-colouring} of $G$. 
Note that we do not require adjacent vertices to have different colours. 
Let us recall the definition of a perfect colouring and make ist precise.
\begin{definition}
A colouring of the vertex set $V$ of some graph $G=(V,E)$ with $m$ colours
is called \emph{perfect} if (1) all colours are used, and (2) for all $i,j$ 
the number of neighbours of colours $j$ of any vertex $v$ of colour $i$ is 
a constant $a_{ij}$. The matrix $A=(a_{ij})_{1 \le i,j \le m}$ is called the
\emph{colour adjacency matrix} of the perfect colouring.
\end{definition}
Two trivial cases are $m=1$, and $m=|V|$. In the latter case the colour 
adjacency matrix equals the adjacency matrix of $G$. 

An early study of perfect colourings is \cite{Sa}, where
the colour adjacency matrix was introduced to study spectral properties of
certain graphs. In particular, the following result was shown in 
\cite[Theorem 4.5]{CDS}, see also \cite[Theorem 9.3.3]{GR}.
\begin{theorem} \label{thm:eigenv}
Let $M$ be the adjacency matrix of some graph $G$ and let $A$ be the
colour adjacency matrix of some perfect colouring of $G$. Then the
characteristic polynomial of $A$ divides the characteristic polynomial
of $M$. In particular, each eigenvalue of $A$ is an eigenvalue of $M$
(with multiplicities).
\end{theorem} 
One main result of \cite{DF} was the following one.
\begin{theorem} \label{thm:charmat}
Suppose $A=(a_{ij}) \in \N^{m \times m}$. 
Then $A$ is a colour adjacency matrix for a perfect $m$-colouring of some graph 
$G=(V,E)$ if and only if the following hold:
\begin{enumerate}
\item \textnormal{(}Weak symmetry\textnormal{)} For all $1 \leq i, j \leq m$
holds: $a_{ij}=0$ if and only if $a_{ji}=0$.
\item \textnormal{(}Consistency\textnormal{)} For any nontrivial cycle 
$(n_{1} \, n_{2} \, \ldots \, n_{t})$ in the symmetric group $S_{m}$ on the 
set $\{1,2,\ldots,m\}$ holds:
\[a_{n_{1},n_{2}}a_{n_{2},n_{3}}\cdots a_{n_{t-1},n_{t}} a_{n_{t},n_{1}} = 
a_{n_{2},n_{1}}a_{n_{3},n_{2}}\cdots a_{n_{t},n_{t-1}} a_{n_{1},n_{t}}.\]
\end{enumerate}
Moreover, there is a connected graph $G$ with a perfect colouring 
corresponding to $A$ if and only if $A$ fulfills (1) and (2), and $A$ is 
irreducible. 
\end{theorem}
A symmetric matrix $M$ is called irreducible if it is not conjugate via a 
permutation matrix to a block diagonal matrix having more than one block.
(By ``block diagonal matrix'' we mean a square matrix having square matrices
on its main diagonal, and all other entries being zero.) It is well-known 
that a directed graph $G$ is connected if and only if its adjacency matrix is
irreducible. A weaker statement is true here: if a graph $G$ is connected 
then its colour adjacency matrix is irreducible. (Because one can travel
from any colour to any other colour.) The generalization of the theorem
to non-connected graphs is straight forward.

The proof of Theorem \ref{thm:charmat} is not hard. It is easy to see
that (1) and (2) are necessary conditions for adjacency matrices for perfect 
colourings. (2) is due to an elaborate double counting argument.
The sufficiency part of the proof is achieved by providing a construction 
that yields a graph with a perfect $m$-colouring for any given matrix satisfying 
the conditions of Theorem \ref{thm:charmat}.

In \cite{DF} we proceeded by using Theorem \ref{thm:charmat} 
to obtain the lists of all adjacency matrices for all perfect 
colourings of any regular graph of degree 3, 4, or 5 with two, 
three or four colours, respectively. These computations were carried
out in \texttt{sagemath} 
\cite{S}. The code and the lists are available online \cite{sws}. 
(Unfortunately, newer versions of \texttt{sagemath} seem to be unable to 
open these. Still you can see the code and the lists as plain text, 
and may import the lists by copy-and-paste into a sagemath notebook.) 
In \cite{DF} we then used Theorem \ref{thm:eigenv} to determine  those
matrices in the lists whose eigenvalues match the eigenvalues of
the edge graphs of the Platonic solids.

For the present paper we follow the same plan. We use Theorem \ref{thm:eigenv}
to determine all matrices in the lists mentioned above whether their 
eigenvalues match the spectrum of the edge graphs of
the $d$-simplex, respectively the $d$-cube. This is easily
done in \texttt{sagemath}. We end up with a list of candidates for perfect 
colourings. This is done in the next section. It remains to show the
existence, respectively non-existence, of a perfect colouring for each of the 
candidates, which is done in Sections \ref{sec:colsimplex}, \ref{sec:col4cube},
and \ref{sec:col5cube}.

\section{Colour adjacency matrices of hypercubes and simplices}
\label{sec:candidates}
As just mentioned, we extract from the respective list those matrices whose 
eigenvalues are in the spectrum of the the 4-simplex, the 5-simplex, the 4-cube, 
or the 5-cube, respectively. The eigenvalues of 
these graphs are given in Table \ref{tab:eigen}. These values
can be found for instance in \cite{CDS}. An entry $a^n$ means that
$a$ is an eigenvalue of algebraic multiplicity $n$.

\begin{table}[h]
\[ \begin{array}{l|l}
\hspace{1cm} G & \hspace{1.5cm} \mbox{eigenvalues} \\
\hline
\mbox{4-simplex} & -1^4,4 \\ 
\mbox{5-simplex} & -1^5,5 \\
\mbox{4-cube} & -4,-2^4,0^6,2^4,4 \\ 
\mbox{5-cube} & -5, -3^5, -1^{10}, 1^{10}, 3^5, 5\\
\end{array}  \]
\caption{The eigenvalues of the graphs of the $d$-simplex and the 
$d$-cube for $d\in \{4,5\}$. A superscript denotes the multiplicity of 
the respective eigenvalue. \label{tab:eigen}}
\end{table}

We used  \texttt{sagemath} to check which of the matrices 
in the corresponding list have eigenvalues in the respective spectrum of
the graph under consideration. The lists of these candidates are
given in Sections \ref{sec:colsimplex}, \ref{sec:col4cube}, and 
\ref{sec:col5cube} below. Table \ref{tab:compare} gives a summary by
providing the number of these candidates in each case.

\begin{table}[h]
\begin{center}
\begin{tabular}{c|cccc}
$m$ \textbackslash \, $G$ & 4-simplex & 5-simplex & 4-cube & 5-cube \\
\hline
2 & 2/2 & 3/3 & 5/6 & 6/9\\
3 & 2/2 & 3/3 & 5/10 & 8/16\\
4 & 1/1 & 2/2 & ?/23 & ?/57 \\
5 & 1 & 1 & 1 & ?\\
6 & 0 & 1 & 1 & ? 
\end{tabular}
\end{center}
\caption{The number of actual perfect colourings compared to the number
of possible colour adjacency matrices (that is, correct eigenvalues) 
for perfect $m$-colourings of the 4-simplex, 5-simplex, 4-cube, 5-cube, 
respectively. \label{tab:compare}}
\end{table}

The table is to be read as follows. Each entry in row 5 and 6 just
shows the number of perfect colourings. An entry in rows 1 to 4
is of the form $a/b$, where $b$ is the number of candidates and $a$ 
is the number of actual perfect colourings.
For instance, the entry 5/10 (4-cube with 3 colours) in the table tells us:
among all 64 matrices that are possible for perfect 3-colourings
of 4-regular graphs (see \cite{DF} or \cite{sws}) only 10 have all their 
eigenvalues in $\{-4,-2,0,2,4\}$. These are the candidates
for perfect 3-colourings of the 4-cube. Only 5 of them correspond
to actual perfect colourings of the 4-cube, as we will see in Section
\ref{sec:col4cube}.

\section{All perfect colourings of the 4-simplex and the 5-simplex} \label{sec:colsimplex}
By the procedure described above we obtain the following matrices
as candidates for perfect colourings of the 4-simplex and the 5-simplex. 
\begin{enumerate}
\item 2 colours:
$\left(\begin{smallmatrix}
0 & 4 \\
1 & 3
\end{smallmatrix}\right), \left(\begin{smallmatrix}
1 & 3 \\
2 & 2
\end{smallmatrix}\right)$
\item 3 colours:
$\left(\begin{smallmatrix}
0 & 1 & 3 \\
1 & 0 & 3 \\
1 & 1 & 2
\end{smallmatrix}\right), \left(\begin{smallmatrix}
0 & 2 & 2 \\
1 & 1 & 2 \\
1 & 2 & 1
\end{smallmatrix}\right)$
\item 4 colours: $ \left(\begin{smallmatrix}
0 & 1 & 1 & 2 \\
1 & 0 & 1 & 2 \\
1 & 1 & 0 & 2 \\
1 & 1 & 1 & 1
\end{smallmatrix}\right)$
\end{enumerate}

5-simplex:
\begin{enumerate}
\item 2 colours:
$\left(\begin{smallmatrix}
0 & 5 \\
1 & 4
\end{smallmatrix}\right), \left(\begin{smallmatrix}
1 & 4 \\
2 & 3
\end{smallmatrix}\right), \left(\begin{smallmatrix}
2 & 3 \\
3 & 2
\end{smallmatrix}\right)$

\item 3 colours:
$\left(\begin{smallmatrix}
0 & 1 & 4 \\
1 & 0 & 4 \\
1 & 1 & 3
\end{smallmatrix}\right), \left(\begin{smallmatrix}
0 & 2 & 3 \\
1 & 1 & 3 \\
1 & 2 & 2
\end{smallmatrix}\right), \left(\begin{smallmatrix}
1 & 2 & 2 \\
2 & 1 & 2 \\
2 & 2 & 1
\end{smallmatrix}\right)$
\item 4 colours: 
$\left(\begin{smallmatrix}
0 & 1 & 1 & 3 \\
1 & 0 & 1 & 3 \\
1 & 1 & 0 & 3 \\
1 & 1 & 1 & 2
\end{smallmatrix}\right), \left(\begin{smallmatrix}
0 & 1 & 2 & 2 \\
1 & 0 & 2 & 2 \\
1 & 1 & 1 & 2 \\
1 & 1 & 2 & 1
\end{smallmatrix}\right)$

\end{enumerate}

In principle it remains to find a perfect colouring of the simplex 
for each of these matrices. However, since the edge graph of the 
$d$-simplex is just the complete graph $K_{d+1}$, it gets much simpler. 
Because of the following result all these candidates correspond 
to perfect actual colourings. 
\begin{lemma}
 Any colouring of $K_{d+1}$ is perfect.
Any integer partition of $d+1$ into $m$ summands gives a perfect
$m$-colouring of $K_{d+1}$, and vice versa.
\end{lemma}
\begin{proof}
This is immediate. Since each vertex in $K_{d+1}$ has all other vertices
as neighbours there is no further restriction imposed by the 
combinatorics of the neighbours. Indeed, any arbitrary $m$-colouring
of the vertices just corresponds to any partition of the vertex set $V$
into $m$ sets of size $a_1, \ldots, a_m$. Then $a_1+\cdots+a_m=d+1$. 
\end{proof}
For instance, the three perfect colourings of the 5-simplex with 
two colours are in one-to-one correspondence with the integer
partition of 6 into two summands: $6=5+1=4+2=3+3$, the 
three perfect colourings of the 5-simplex with 
three colours are in one-to-one correspondence with the integer
partition of 6 into three summands: $6=4+1+1=3+2+1=2+2+2$, and
the two perfect colourings of the 5-simplex with 
four colours are in one-to-one correspondence with the integer
partition of 6 into four summands: $6=3+1+1+1=2+2+1+1$.

This observation allows us to determine all other perfect colourings 
of the 4-simplex and 5-simplex with five or more colours in Table
\ref{tab:compare}. Of course there cannot be more colours than
vertices, hence Table \ref{tab:compare} contains all perfect
colourings of the 4-simplex and the 5-simplex. The only remaining 
case in the list that is not entirely trivial is to colour the 5-simplex 
(having six vertices) with five colours. This is obtained from the
integer partition $6=2+1+1+1+1$.

Having said all this all the computations in \texttt{sagemath} 
yielding the matrices above seem unnecessary. Still we carried 
them out to serve as a sanity check of the software.

\section{Perfect colourings of the 4-hypercube} \label{sec:col4cube}
Applying the analogous procedure we obtain a list of all candidates for
colour adjacency matrices for 2-, 3-, and 4-colourings of the 4-cube, 
respectively. Here they are:

\begin{enumerate}
\item 2 colours:
$\left(\begin{smallmatrix}
0 & 4 \\
2 & 2
\end{smallmatrix}\right), \left(\begin{smallmatrix}
0 & 4 \\
4 & 0
\end{smallmatrix}\right), \left(\begin{smallmatrix}
1 & 3 \\
1 & 3
\end{smallmatrix}\right), \left(\begin{smallmatrix}
1 & 3 \\
3 & 1
\end{smallmatrix}\right), \left(\begin{smallmatrix}
2 & 2 \\
2 & 2
\end{smallmatrix}\right), \left(\begin{smallmatrix}
3 & 1 \\
1 & 3
\end{smallmatrix}\right)$

\item 3 colours:
$\left(\begin{smallmatrix}
0 & 0 & 4 \\
0 & 0 & 4 \\
1 & 1 & 2
\end{smallmatrix}\right), \left(\begin{smallmatrix}
0 & 0 & 4 \\
0 & 0 & 4 \\
1 & 3 & 0
\end{smallmatrix}\right), \left(\begin{smallmatrix}
0 & 0 & 4 \\
0 & 0 & 4 \\
2 & 2 & 0
\end{smallmatrix}\right), 
\left(\begin{smallmatrix}
0 & 2 & 2 \\
2 & 0 & 2 \\
1 & 1 & 2
\end{smallmatrix}\right), \left(\begin{smallmatrix}
0 & 2 & 2 \\
2 & 0 & 2 \\
2 & 2 & 0
\end{smallmatrix}\right), \left(\begin{smallmatrix}
0 & 2 & 2 \\
2 & 1 & 1 \\
2 & 1 & 1
\end{smallmatrix}\right), \left(\begin{smallmatrix}
0 & 2 & 2 \\
2 & 2 & 0 \\
2 & 0 & 2
\end{smallmatrix}\right)\\ \left(\begin{smallmatrix}
1 & 1 & 2 \\
1 & 1 & 2 \\
1 & 1 & 2
\end{smallmatrix}\right), \left(\begin{smallmatrix}
2 & 0 & 2 \\
0 & 2 & 2 \\
1 & 1 & 2
\end{smallmatrix}\right), \left(\begin{smallmatrix}
2 & 2 & 0 \\
1 & 0 & 3 \\
0 & 2 & 2
\end{smallmatrix}\right)$
\item 4 colours:
$\left(\begin{smallmatrix}
0 & 0 & 0 & 4 \\
0 & 0 & 0 & 4 \\
0 & 0 & 0 & 4 \\
1 & 1 & 2 & 0
\end{smallmatrix}\right), \left(\begin{smallmatrix}
0 & 0 & 0 & 4 \\
0 & 0 & 4 & 0 \\
0 & 2 & 0 & 2 \\
2 & 0 & 2 & 0
\end{smallmatrix}\right), \left(\begin{smallmatrix}
0 & 0 & 1 & 3 \\
0 & 0 & 1 & 3 \\
1 & 1 & 2 & 0 \\
1 & 1 & 0 & 2
\end{smallmatrix}\right), \left(\begin{smallmatrix}
0 & 0 & 1 & 3 \\
0 & 0 & 3 & 1 \\
1 & 3 & 0 & 0 \\
3 & 1 & 0 & 0
\end{smallmatrix}\right), \left(\begin{smallmatrix}
0 & 0 & 2 & 2 \\
0 & 0 & 2 & 2 \\
1 & 1 & 0 & 2 \\
1 & 1 & 2 & 0
\end{smallmatrix}\right), \left(\begin{smallmatrix}
0 & 0 & 2 & 2 \\
0 & 0 & 2 & 2 \\
1 & 1 & 1 & 1 \\
1 & 1 & 1 & 1
\end{smallmatrix}\right),\\ \left(\begin{smallmatrix}
0 & 0 & 2 & 2 \\
0 & 0 & 2 & 2 \\
1 & 1 & 2 & 0 \\
1 & 1 & 0 & 2
\end{smallmatrix}\right), \left(\begin{smallmatrix}
0 & 0 & 2 & 2 \\
0 & 0 & 2 & 2 \\
2 & 2 & 0 & 0 \\
2 & 2 & 0 & 0
\end{smallmatrix}\right), \left(\begin{smallmatrix}
0 & 1 & 0 & 3 \\
1 & 0 & 3 & 0 \\
0 & 1 & 0 & 3 \\
1 & 0 & 3 & 0
\end{smallmatrix}\right), \left(\begin{smallmatrix}
0 & 1 & 0 & 3 \\
1 & 0 & 3 & 0 \\
0 & 1 & 2 & 1 \\
1 & 0 & 1 & 2
\end{smallmatrix}\right), \left(\begin{smallmatrix}
0 & 1 & 1 & 2 \\
1 & 0 & 2 & 1 \\
1 & 2 & 0 & 1 \\
2 & 1 & 1 & 0
\end{smallmatrix}\right), \left(\begin{smallmatrix}
0 & 1 & 1 & 2 \\
1 & 1 & 1 & 1 \\
1 & 1 & 1 & 1 \\
2 & 1 & 1 & 0
\end{smallmatrix}\right), \left(\begin{smallmatrix}
0 & 1 & 1 & 2 \\
1 & 2 & 0 & 1 \\
1 & 0 & 2 & 1 \\
2 & 1 & 1 & 0
\end{smallmatrix}\right),\\ \left(\begin{smallmatrix}
0 & 2 & 0 & 2 \\
2 & 0 & 0 & 2 \\
0 & 0 & 2 & 2 \\
1 & 1 & 2 & 0
\end{smallmatrix}\right), \left(\begin{smallmatrix}
0 & 2 & 2 & 0 \\
1 & 0 & 0 & 3 \\
1 & 0 & 0 & 3 \\
0 & 2 & 2 & 0
\end{smallmatrix}\right), \left(\begin{smallmatrix}
1 & 0 & 0 & 3 \\
0 & 1 & 3 & 0 \\
0 & 1 & 1 & 2 \\
1 & 0 & 2 & 1
\end{smallmatrix}\right), \left(\begin{smallmatrix}
1 & 0 & 1 & 2 \\
0 & 1 & 2 & 1 \\
1 & 2 & 1 & 0 \\
2 & 1 & 0 & 1
\end{smallmatrix}\right), \left(\begin{smallmatrix}
1 & 0 & 3 & 0 \\
0 & 1 & 0 & 3 \\
1 & 0 & 0 & 3 \\
0 & 1 & 1 & 2
\end{smallmatrix}\right), \left(\begin{smallmatrix}
1 & 1 & 0 & 2 \\
1 & 1 & 0 & 2 \\
0 & 0 & 2 & 2 \\
1 & 1 & 2 & 0
\end{smallmatrix}\right), \left(\begin{smallmatrix}
1 & 1 & 1 & 1 \\
1 & 1 & 1 & 1 \\
1 & 1 & 1 & 1 \\
1 & 1 & 1 & 1
\end{smallmatrix}\right),\\ \left(\begin{smallmatrix}
1 & 1 & 1 & 1 \\
1 & 1 & 1 & 1 \\
1 & 1 & 2 & 0 \\
1 & 1 & 0 & 2
\end{smallmatrix}\right), \left(\begin{smallmatrix}
2 & 0 & 0 & 2 \\
0 & 2 & 0 & 2 \\
0 & 0 & 2 & 2 \\
1 & 1 & 2 & 0
\end{smallmatrix}\right), \left(\begin{smallmatrix}
2 & 0 & 1 & 1 \\
0 & 2 & 1 & 1 \\
1 & 1 & 2 & 0 \\
1 & 1 & 0 & 2
\end{smallmatrix}\right)$
\end{enumerate}
Unlike in the previous section not all of these candidates
give rise to actual perfect colourings. 
In the following result we list all of the matrices
that  correspond to actual perfect colourings of the
4-cube  with two and three colours. We leave to check
the existence of the 23 distinct 4-colourings 
as a challenge to the reader. 
\begin{theorem} \label{thm:4cube}
All perfect 2-colourings of the 4-cube 
are the five ones corresponding to
\[ \left(\begin{smallmatrix}
0 & 4 \\
4 & 0
\end{smallmatrix}\right), \left(\begin{smallmatrix}
1 & 3 \\
1 & 3
\end{smallmatrix}\right), \left(\begin{smallmatrix}
1 & 3 \\
3 & 1
\end{smallmatrix}\right), \left(\begin{smallmatrix}
2 & 2 \\
2 & 2
\end{smallmatrix}\right), \left(\begin{smallmatrix}
3 & 1 \\
1 & 3
\end{smallmatrix}\right) \]
All perfect 3-colourings of the 4-cube 
are the five ones corresponding to
\[ \left(\begin{smallmatrix}
0 & 0 & 4 \\
0 & 0 & 4 \\
1 & 3 & 0
\end{smallmatrix}\right), \left(\begin{smallmatrix}
0 & 0 & 4 \\
0 & 0 & 4 \\
2 & 2 & 0
\end{smallmatrix}\right), 
\left(\begin{smallmatrix}
0 & 2 & 2 \\
2 & 0 & 2 \\
1 & 1 & 2
\end{smallmatrix}\right),  \left(\begin{smallmatrix}
1 & 1 & 2 \\
1 & 1 & 2 \\
1 & 1 & 2
\end{smallmatrix}\right), \left(\begin{smallmatrix}
2 & 0 & 2 \\
0 & 2 & 2 \\
1 & 1 & 2
\end{smallmatrix}\right). \]
\end{theorem}
\begin{proof} Perfect colourings corresponding to the colour adjacency 
matrices above are shown in Figure \ref{fig:4cube}. It remains to
show that there are no further perfect colourings of the 4-cube corresponding 
to the other candidates.

\textbf{Two colours}, matrix $\left(\begin{smallmatrix} 0 & 4 \\ 2 & 2
\end{smallmatrix}\right)$: This is not possible, as one can
easily see as follows: a white vertex (without loss of generality 
vertex 0 in this diagram)
\[ \includegraphics[width=.40\textwidth]{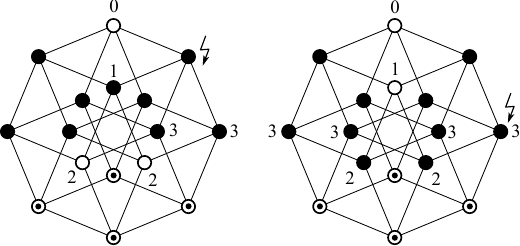} \]
has four black neighbours. If the vertex labelled 1 is black
then the two vertices labelled 2 are necessarily white. 
This in turn forces the vertices labelled 3 to be black. 
Now the black vertex labelled with a lightning bolt has three black 
neighbours. Contradiction. 

If vertex 1 is white, then the vertices labelled 2 are black.
Moreover, the topmost two black vertices already have
two white neighbours, hence all vertices labelled 3 are necessarily 
black. Now the black vertex labelled with a lightning bolt has three
black neighbours. Contradiction.

\textbf{Three colours}: we need to exclude the five matrices
\[  \left(\begin{smallmatrix}
0 & 0 & 4 \\
0 & 0 & 4 \\
1 & 1 & 2
\end{smallmatrix}\right), \left(\begin{smallmatrix}
0 & 2 & 2 \\
2 & 0 & 2 \\
2 & 2 & 0
\end{smallmatrix}\right), 
\left(\begin{smallmatrix}
0 & 2 & 2 \\
2 & 1 & 1 \\
2 & 1 & 1
\end{smallmatrix}\right),  \left(\begin{smallmatrix}
0 & 2 & 2 \\
2 & 2 & 0 \\
2 & 0 & 2
\end{smallmatrix}\right), \left(\begin{smallmatrix}
2 & 2 & 0 \\
1 & 0 & 3 \\
0 & 2 & 2
\end{smallmatrix}\right). \]
This is easy: For instance in a perfect 3-colouring corresponding
to $\left(\begin{smallmatrix}
0 & 0 & 4 \\
0 & 0 & 4 \\
1 & 1 & 2
\end{smallmatrix}\right)$ we can identify the colours number 1 and 2
(white and black). We obtain a perfect 2-colouring with colour adjacency matrix
$\left(\begin{smallmatrix}
0 & 4 \\
2 & 2
\end{smallmatrix}\right)$. But we just saw that such a perfect
2-colouring does not exist. All remaining matrices can be excluded
in this manner. Note that for the last one we need to identify colours 
1 and 3 in order to obtain $\left(\begin{smallmatrix}
2 & 2 \\
4 & 0
\end{smallmatrix}\right)$, which is just a permutation of 
$\left(\begin{smallmatrix}
0 & 4 \\
2 & 2
\end{smallmatrix}\right)$, respectively a permutation of colours, hence 
also impossible.
\end{proof}

\begin{figure}
\[ \includegraphics[width=.9\textwidth]{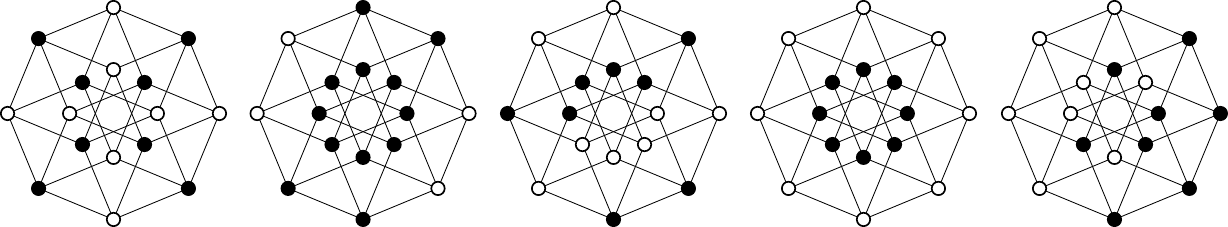} \]
\[ \includegraphics[width=.9\textwidth]{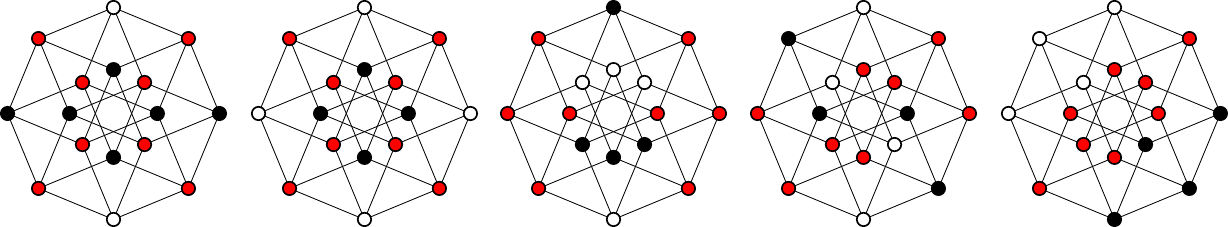} \]
\caption{The perfect colourings of the 4-cube with two and three colours.
\label{fig:4cube}}
\end{figure}

\section{Perfect colourings of the 5-cube} \label{sec:col5cube}
Applying the analogous procedure we obtain a list of all candidates
for colour adjacency matrices for 2-, 3-, and 4-colourings of the 5-cube,
respectively. 
\begin{enumerate}
\item 2 colours
$\left(\begin{smallmatrix}
0 & 5 \\
1 & 4
\end{smallmatrix}\right), \left(\begin{smallmatrix}
0 & 5 \\
3 & 2
\end{smallmatrix}\right), \left(\begin{smallmatrix}
0 & 5 \\
5 & 0
\end{smallmatrix}\right), \left(\begin{smallmatrix}
1 & 4 \\
2 & 3
\end{smallmatrix}\right), \left(\begin{smallmatrix}
1 & 4 \\
4 & 1
\end{smallmatrix}\right), \left(\begin{smallmatrix}
2 & 3 \\
1 & 4
\end{smallmatrix}\right), \left(\begin{smallmatrix}
2 & 3 \\
3 & 2
\end{smallmatrix}\right), \left(\begin{smallmatrix}
3 & 2 \\
2 & 3
\end{smallmatrix}\right), \left(\begin{smallmatrix}
4 & 1 \\
1 & 4
\end{smallmatrix}\right)$
\item 3 colours
$\left(\begin{smallmatrix}
0 & 1 & 4 \\
1 & 0 & 4 \\
1 & 1 & 3
\end{smallmatrix}\right), \left(\begin{smallmatrix}
0 & 1 & 4 \\
1 & 0 & 4 \\
2 & 2 & 1
\end{smallmatrix}\right), \left(\begin{smallmatrix}
0 & 2 & 3 \\
1 & 1 & 3 \\
1 & 2 & 2
\end{smallmatrix}\right), \left(\begin{smallmatrix}
0 & 2 & 3 \\
2 & 1 & 2 \\
3 & 2 & 0
\end{smallmatrix}\right), \left(\begin{smallmatrix}
0 & 3 & 2 \\
3 & 0 & 2 \\
1 & 1 & 3
\end{smallmatrix}\right), \left(\begin{smallmatrix}
0 & 5 & 0 \\
1 & 0 & 4 \\
0 & 2 & 3
\end{smallmatrix}\right), \left(\begin{smallmatrix}
1 & 0 & 4 \\
0 & 1 & 4 \\
1 & 1 & 3
\end{smallmatrix}\right), \left(\begin{smallmatrix}
1 & 0 & 4 \\
0 & 1 & 4 \\
1 & 3 & 1
\end{smallmatrix}\right),$ $\left(\begin{smallmatrix}
1 & 0 & 4 \\
0 & 1 & 4 \\
2 & 2 & 1
\end{smallmatrix}\right), \left(\begin{smallmatrix}
1 & 2 & 2 \\
2 & 1 & 2 \\
1 & 1 & 3
\end{smallmatrix}\right), \left(\begin{smallmatrix}
1 & 2 & 2 \\
2 & 1 & 2 \\
2 & 2 & 1
\end{smallmatrix}\right), \left(\begin{smallmatrix}
1 & 2 & 2 \\
2 & 2 & 1 \\
2 & 1 & 2
\end{smallmatrix}\right), \left(\begin{smallmatrix}
1 & 2 & 2 \\
2 & 3 & 0 \\
2 & 0 & 3
\end{smallmatrix}\right), \left(\begin{smallmatrix}
2 & 1 & 2 \\
1 & 2 & 2 \\
1 & 1 & 3
\end{smallmatrix}\right), \left(\begin{smallmatrix}
3 & 0 & 2 \\
0 & 3 & 2 \\
1 & 1 & 3
\end{smallmatrix}\right), \left(\begin{smallmatrix}
3 & 2 & 0 \\
1 & 1 & 3 \\
0 & 2 & 3
\end{smallmatrix}\right)$
\item 4 colours: see Appendix B.
\end{enumerate}
In the following result we list all of the matrices
that  correspond to actual perfect colourings of the
4-cube with two and three colours. We leave to check
the existence of the 57 distinct 4-colourings to the reader
as a challenging exercise.
\begin{theorem} \label{thm:5cube}
All perfect 2-colourings of the 5-cube 
are the six ones corresponding to
\[ \left(\begin{smallmatrix}
0 & 5 \\
5 & 0
\end{smallmatrix}\right), \left(\begin{smallmatrix}
1 & 4 \\
4 & 1
\end{smallmatrix}\right), \left(\begin{smallmatrix}
2 & 3 \\
1 & 4
\end{smallmatrix}\right), \left(\begin{smallmatrix}
2 & 3 \\
3 & 2
\end{smallmatrix}\right), \left(\begin{smallmatrix}
3 & 2 \\
2 & 3
\end{smallmatrix}\right), \left(\begin{smallmatrix}
4 & 1 \\
1 & 4
\end{smallmatrix}\right). \]
All perfect 3-colourings of the 5-cube 
are the eight ones corresponding to
\[ \left(\begin{smallmatrix}
0 & 1 & 4 \\
1 & 0 & 4 \\
2 & 2 & 1
\end{smallmatrix}\right), \left(\begin{smallmatrix}
0 & 3 & 2 \\
3 & 0 & 2 \\
1 & 1 & 3
\end{smallmatrix}\right), 
\left(\begin{smallmatrix}
0 & 5 & 0 \\
1 & 0 & 4 \\
0 & 2 & 3
\end{smallmatrix}\right),
\left(\begin{smallmatrix}
1 & 0 & 4 \\
0 & 1 & 4 \\
1 & 3 & 1
\end{smallmatrix}\right),
\left(\begin{smallmatrix}
1 & 0 & 4 \\
0 & 1 & 4 \\
2 & 2 & 1
\end{smallmatrix}\right),
\left(\begin{smallmatrix}
1 & 2 & 2 \\
2 & 1 & 2 \\
1 & 1 & 3
\end{smallmatrix}\right), \left(\begin{smallmatrix}
2 & 1 & 2 \\
1 & 2 & 2 \\
1 & 1 & 3
\end{smallmatrix}\right), \left(\begin{smallmatrix}
3 & 0 & 2 \\
0 & 3 & 2 \\
1 & 1 & 3
\end{smallmatrix}\right).
\]
\end{theorem}
\begin{proof}
\textbf{Two colours:}
It can be fun to find the particular perfect colourings by
hand for each case. But the following argument proves the existence 
in a simpler manner.

Consider two copies of a perfect colouring of a 3-cube and draw 
an edge between corresponding vertices. Here is an example
for the perfect 2-colouring with colour adjacency matrix 
$\left(\begin{smallmatrix} 0 & 3 \\ 1 & 2
\end{smallmatrix}\right)$. 
\[ \includegraphics[width=.3\textwidth]{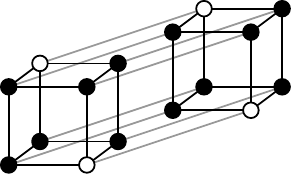} \]
The result is a perfect 2-colouring of the 4-cube with
colour adjacency matrix $\left(\begin{smallmatrix} 1 & 3 \\ 1 & 3
\end{smallmatrix}\right)$. In a similar manner, each
perfect $m$-colouring of a $d$-cube with  colour adjacency matrix 
$A$ yields a $m$-perfect colouring of a $(d+1)$-cube with colour 
adjacency matrix $A+I$, where $I$ is the identity matrix.

In the same way each of the five perfect 2-colourings of the
4-cube gives rise to one perfect 2-colouring of the 5-cube. 
The only remaining case is the matrix $\left(\begin{smallmatrix}
0 & 5 \\ 5 & 0 \end{smallmatrix}\right)$. The corresponding 
perfect 2-colouring comes from the fact that the $d$-cube is a 
bipartite graph for any $d$.

It remains to show that the three other matrices in the list
of candidates above do not yield perfect colourings. In order to 
exclude $\left(\begin{smallmatrix} 0 & 5 \\ 1 & 4
\end{smallmatrix}\right)$ and $\left(\begin{smallmatrix}
1 & 4 \\ 2 & 3 \end{smallmatrix}\right)$ we apply a result from
\cite{DF}: 
\begin{lemma}\label{lem:vercount2}
Let $A= (a_{ij}) \in \N^{2 \times 2}$ be a colour adjacency matrix of 
some connected graph $G=(V,E)$. Then $a_{12}$ and $a_{21}$ are both 
nonzero, and if $v_{i}$ denotes the number of vertices of colour $i$, then
\[ v_{1} = \frac{|V|}{1+\frac{a_{12}}{a_{21}}}, 
v_{2} = \frac{|V|}{\frac{a_{21}}{a_{12}}+1}.  \]
\end{lemma}
Applied to the two matrices above for the particular case
of the 5-cube, where $|V|=32$, we obtain $v_1=\frac{16}{3}$
for $\left(\begin{smallmatrix} 0 & 5 \\ 1 & 4
\end{smallmatrix}\right)$ and $v_1=\frac{32}{3}$ for
$\left(\begin{smallmatrix} 1 & 4 \\ 3 & 2
\end{smallmatrix}\right)$. Since these values are not integers
there is no corresponding perfect colouring of the 5-cube.

For the matrix $\left(\begin{smallmatrix} 0 & 5 \\ 3 & 2
\end{smallmatrix}\right)$ the lemma yields $v_1=12$ and 
$v_2=20$. This does not help to exclude this case.
But it can be excluded as follows:

Without loss of generality let the vertex labelled 1 in the following 
diagram be white.
\[ \includegraphics[width=.65\textwidth]{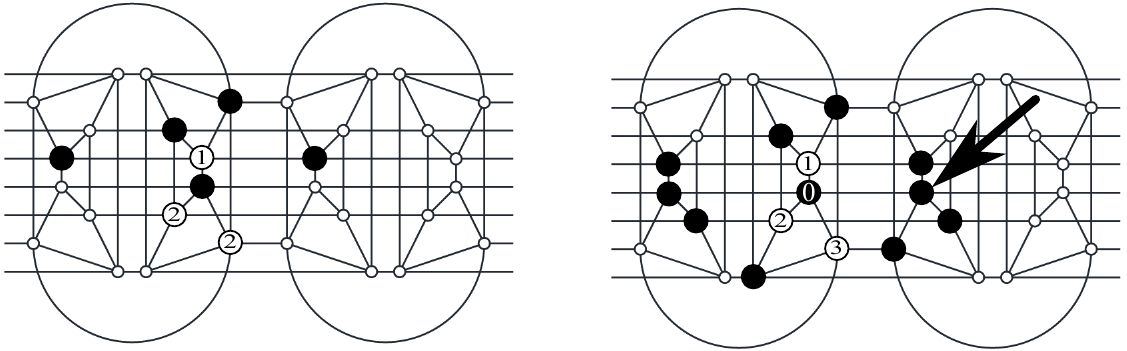} \]
This is indeed the 5-cube graph, see \cite{P}. (The edges that
leave to the left are connected to those that enter on the right 
on the same height. You may imagine the image wrapped on a 
cylinder, or you may draw the missing connections.)
Because of the matrix all five neighbours of 1 have to be black.
Without loss of generality, let the two vertices 2 and 3 be the 
other two white neighbours of the black vertex labelled 0. (Because of
the high symmetry of the 5-cube it does not matter which two are 
chosen.) All neighbours of vertices 2 and 3 are black. The other
two neighbours of 0 are black as well. Now the black vertex labelled 
by an arrow has four black neighbours, contradiction.

\textbf{Three colours:} For the five matrices $\left(\begin{smallmatrix}
1 & 0 & 4 \\
0 & 1 & 4 \\
1 & 3 & 1
\end{smallmatrix}\right),
\left(\begin{smallmatrix}
1 & 0 & 4 \\
0 & 1 & 4 \\
2 & 2 & 1
\end{smallmatrix}\right),
\left(\begin{smallmatrix}
1 & 2 & 2 \\
2 & 1 & 2 \\
1 & 1 & 3
\end{smallmatrix}\right), \left(\begin{smallmatrix}
2 & 1 & 2 \\
1 & 2 & 2 \\
1 & 1 & 3
\end{smallmatrix}\right),$ and $\left(\begin{smallmatrix}
3 & 0 & 2 \\
0 & 3 & 2 \\
1 & 1 & 3
\end{smallmatrix}\right)$, the observation above works:
Each of the five perfect 3-colourings of a 4-cube with colour 
adjacency matrix $A$ yields a perfect 3-colouring of a 5-cube 
with colour adjacency matrix $A+I$.

For two of the remaining matrices, $\left(\begin{smallmatrix}
0 & 1 & 4 \\
1 & 0 & 4 \\
2 & 2 & 1
\end{smallmatrix}\right)$ and $\left(\begin{smallmatrix}
0 & 3 & 2 \\
3 & 0 & 2 \\
1 & 1 & 3
\end{smallmatrix}\right)$, a variation of the same argument
works. Since 
\[ \left(\begin{smallmatrix}
0 & 1 & 4 \\
1 & 0 & 4 \\
2 & 2 & 1
\end{smallmatrix}\right) = \left(\begin{smallmatrix}
0 & 0 & 4 \\
0 & 0 & 4 \\
2 & 2 & 0
\end{smallmatrix}\right) + \left(\begin{smallmatrix}
0 & 1 & 0 \\
1 & 0 & 0 \\
0 & 0 & 1
\end{smallmatrix}\right),  \quad \mbox{ and }  \left(\begin{smallmatrix}
0 & 3 & 2 \\
3 & 0 & 2 \\
1 & 1 & 3
\end{smallmatrix}\right) = \left(\begin{smallmatrix}
0 & 2 & 2 \\
2 & 0 & 2 \\
2 & 1 & 1
\end{smallmatrix}\right) + \left(\begin{smallmatrix}
0 & 1 & 0 \\
1 & 0 & 0 \\
0 & 0 & 1
\end{smallmatrix}\right), \]
two copies of the corresponding perfect 3-colourings of the 4-cube yield 
perfect 3-colourings of the 5-cube. Here we need to use two slightly different
copies: in one colouring the roles of black and white are swapped.
So each white vertex gains a new black neighbour and vice versa.
(This trick works whenever the number of black vertices equals the 
number white of vertices, and when they are distributed in the same 
manner on the cube. Both conditions are fulfilled here, compare
Figure \ref{fig:4cube} bottom row, second and third from left.)

The last matrix $\left(\begin{smallmatrix}
0 & 5 & 0 \\
1 & 0 & 4 \\
0 & 2 & 3
\end{smallmatrix}\right)$ corresponds to the perfect 3-colouring shown 
in Figure \ref{fig:050104023}.

\begin{figure}
\[ \includegraphics[width=.4\textwidth]{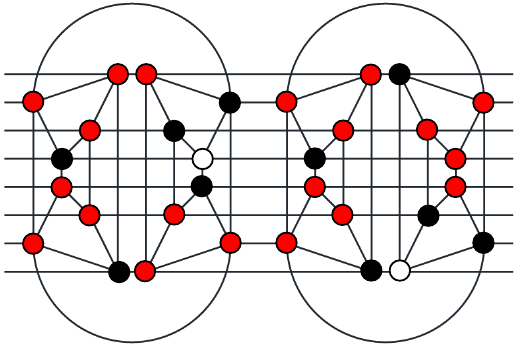} \]
\caption{A perfect 3-colouring of the 5-cube with colour adjacency
matrix $\left(\begin{smallmatrix}
0 & 5 & 0 \\
1 & 0 & 4 \\
0 & 2 & 3
\end{smallmatrix}\right)$. This is indeed the graph of the 5-cube 
if we imagine the image wrapped on a cylinder, so edges that
leave to the left are connected to those that enter on the right 
on the same height.  \label{fig:050104023}}
\end{figure}

The remaining matrices that do not correspond to perfect colourings 
of the 5-cube can all be excluded by the last argument from
the proof of Theorem \ref{thm:4cube}. For instance, let us assume
there is a perfect colouring corresponding to either one of
$ \left(\begin{smallmatrix}
0 & 1 & 4 \\
1 & 0 & 4 \\
1 & 1 & 3
\end{smallmatrix}\right)$ and  $\left(\begin{smallmatrix}
0 & 2 & 3 \\
1 & 1 & 3 \\
1 & 2 & 2
\end{smallmatrix}\right)$. By identifying colours 2 and 3 we obtain
a perfect 2-colouring with colour adjacency matrix 
$\left(\begin{smallmatrix} 0 & 5 \\ 1 & 4
\end{smallmatrix}\right)$. But we saw already that such a perfect 
2-colouring does not exist.

All remaining matrices can be excluded in this way, reducing them
to an impossible perfect 2-colouring with matrix 
$\left(\begin{smallmatrix} 1 & 4 \\ 2 & 3
\end{smallmatrix}\right)$ (or $\left(\begin{smallmatrix} 3 & 2 \\ 4 & 1
\end{smallmatrix}\right)$, which is the same up to swapping colours).
In most remaining cases this is pretty obvious. In the case
$\left(\begin{smallmatrix}
3 & 2 & 0 \\
1 & 1 & 3 \\
0 & 2 & 3
\end{smallmatrix}\right)$ we need to identify colours 1 and 3 in order 
to obtain $\left(\begin{smallmatrix} 3 & 2 \\ 4 & 1
\end{smallmatrix}\right)$.
\end{proof}

\section*{Acknowledgments}
The author thanks the Research Center of Mathematical
Modeling (RCM$^2$) at Bielefeld University for ongoing financial support.
Special thanks to Caya Schubert for providing the perfect colourings of
the 4-cube.
\bibliographystyle{amsplain}

\section*{Appendix: possible colour adjacency matrices 
for perfect 4-colourings of the 5-cube}
All possible colour adjacency matrices $A$ for perfect 4-colourings of the edge
graph of the 5-dimensional cube. It is not known which correspond to actual
perfect colourings. In principle all of them can be studied by the methods
of this paper.

\parindent0em

\medskip

$\left(\begin{smallmatrix}
0 & 0 & 0 & 5 \\
0 & 0 & 3 & 2 \\
0 & 3 & 0 & 2 \\
1 & 2 & 2 & 0
\end{smallmatrix}\right), \left(\begin{smallmatrix}
0 & 0 & 0 & 5 \\
0 & 0 & 3 & 2 \\
0 & 3 & 2 & 0 \\
1 & 2 & 0 & 2
\end{smallmatrix}\right), \left(\begin{smallmatrix}
0 & 0 & 0 & 5 \\
0 & 0 & 5 & 0 \\
0 & 1 & 0 & 4 \\
1 & 0 & 4 & 0
\end{smallmatrix}\right), \left(\begin{smallmatrix}
0 & 0 & 0 & 5 \\
0 & 0 & 5 & 0 \\
0 & 3 & 0 & 2 \\
3 & 0 & 2 & 0
\end{smallmatrix}\right), \left(\begin{smallmatrix}
0 & 0 & 0 & 5 \\
0 & 1 & 2 & 2 \\
0 & 2 & 1 & 2 \\
1 & 2 & 2 & 0
\end{smallmatrix}\right), \left(\begin{smallmatrix}
0 & 0 & 0 & 5 \\
0 & 2 & 1 & 2 \\
0 & 1 & 2 & 2 \\
1 & 2 & 2 & 0
\end{smallmatrix}\right), \left(\begin{smallmatrix}
0 & 0 & 0 & 5 \\
0 & 3 & 0 & 2 \\
0 & 0 & 3 & 2 \\
1 & 2 & 2 & 0
\end{smallmatrix}\right), \left(\begin{smallmatrix}
0 & 0 & 1 & 4 \\
0 & 0 & 4 & 1 \\
1 & 4 & 0 & 0 \\
4 & 1 & 0 & 0
\end{smallmatrix}\right)$, $\left(\begin{smallmatrix}
0 & 0 & 2 & 3 \\
0 & 0 & 3 & 2 \\
2 & 3 & 0 & 0 \\
3 & 2 & 0 & 0
\end{smallmatrix}\right), \left(\begin{smallmatrix}
0 & 0 & 5 & 0 \\
0 & 0 & 2 & 3 \\
1 & 1 & 0 & 3 \\
0 & 1 & 2 & 2
\end{smallmatrix}\right), \left(\begin{smallmatrix}
0 & 0 & 5 & 0 \\
0 & 3 & 2 & 0 \\
1 & 1 & 0 & 3 \\
0 & 0 & 2 & 3
\end{smallmatrix}\right), \left(\begin{smallmatrix}
0 & 1 & 0 & 4 \\
1 & 0 & 0 & 4 \\
0 & 0 & 1 & 4 \\
1 & 1 & 2 & 1
\end{smallmatrix}\right), \left(\begin{smallmatrix}
0 & 1 & 0 & 4 \\
1 & 0 & 4 & 0 \\
0 & 2 & 0 & 3 \\
2 & 0 & 3 & 0
\end{smallmatrix}\right), \left(\begin{smallmatrix}
0 & 1 & 0 & 4 \\
1 & 0 & 4 & 0 \\
0 & 2 & 2 & 1 \\
2 & 0 & 1 & 2
\end{smallmatrix}\right), \left(\begin{smallmatrix}
0 & 1 & 1 & 3 \\
1 & 0 & 1 & 3 \\
1 & 1 & 0 & 3 \\
1 & 1 & 1 & 2
\end{smallmatrix}\right), \left(\begin{smallmatrix}
0 & 1 & 1 & 3 \\
1 & 0 & 1 & 3 \\
1 & 1 & 3 & 0 \\
1 & 1 & 0 & 3
\end{smallmatrix}\right)$, $\left(\begin{smallmatrix}
0 & 1 & 1 & 3 \\
1 & 0 & 3 & 1 \\
1 & 3 & 0 & 1 \\
3 & 1 & 1 & 0
\end{smallmatrix}\right), \left(\begin{smallmatrix}
0 & 1 & 1 & 3 \\
1 & 1 & 0 & 3 \\
1 & 0 & 1 & 3 \\
1 & 1 & 1 & 2
\end{smallmatrix}\right), \left(\begin{smallmatrix}
0 & 1 & 1 & 3 \\
1 & 1 & 2 & 1 \\
1 & 2 & 1 & 1 \\
3 & 1 & 1 & 0
\end{smallmatrix}\right), \left(\begin{smallmatrix}
0 & 1 & 1 & 3 \\
1 & 2 & 1 & 1 \\
1 & 1 & 2 & 1 \\
3 & 1 & 1 & 0
\end{smallmatrix}\right), \left(\begin{smallmatrix}
0 & 1 & 1 & 3 \\
1 & 3 & 0 & 1 \\
1 & 0 & 3 & 1 \\
3 & 1 & 1 & 0
\end{smallmatrix}\right), \left(\begin{smallmatrix}
0 & 1 & 2 & 2 \\
1 & 0 & 2 & 2 \\
1 & 1 & 0 & 3 \\
1 & 1 & 3 & 0
\end{smallmatrix}\right), \left(\begin{smallmatrix}
0 & 1 & 2 & 2 \\
1 & 0 & 2 & 2 \\
1 & 1 & 1 & 2 \\
1 & 1 & 2 & 1
\end{smallmatrix}\right), \left(\begin{smallmatrix}
0 & 1 & 2 & 2 \\
1 & 0 & 2 & 2 \\
1 & 1 & 2 & 1 \\
1 & 1 & 1 & 2
\end{smallmatrix}\right)$, $\left(\begin{smallmatrix}
0 & 1 & 2 & 2 \\
1 & 0 & 2 & 2 \\
1 & 1 & 3 & 0 \\
1 & 1 & 0 & 3
\end{smallmatrix}\right), \left(\begin{smallmatrix}
0 & 1 & 2 & 2 \\
1 & 0 & 2 & 2 \\
2 & 2 & 0 & 1 \\
2 & 2 & 1 & 0
\end{smallmatrix}\right), \left(\begin{smallmatrix}
0 & 1 & 2 & 2 \\
1 & 0 & 2 & 2 \\
2 & 2 & 1 & 0 \\
2 & 2 & 0 & 1
\end{smallmatrix}\right), \left(\begin{smallmatrix}
0 & 2 & 0 & 3 \\
2 & 0 & 3 & 0 \\
0 & 1 & 0 & 4 \\
1 & 0 & 4 & 0
\end{smallmatrix}\right), \left(\begin{smallmatrix}
0 & 2 & 0 & 3 \\
2 & 0 & 3 & 0 \\
0 & 1 & 2 & 2 \\
1 & 0 & 2 & 2
\end{smallmatrix}\right), \left(\begin{smallmatrix}
0 & 3 & 0 & 2 \\
3 & 0 & 0 & 2 \\
0 & 0 & 3 & 2 \\
1 & 1 & 2 & 1
\end{smallmatrix}\right), \left(\begin{smallmatrix}
1 & 0 & 0 & 4 \\
0 & 1 & 0 & 4 \\
0 & 0 & 1 & 4 \\
1 & 1 & 2 & 1
\end{smallmatrix}\right), \left(\begin{smallmatrix}
1 & 0 & 0 & 4 \\
0 & 1 & 4 & 0 \\
0 & 2 & 1 & 2 \\
2 & 0 & 2 & 1
\end{smallmatrix}\right)$, $\left(\begin{smallmatrix}
1 & 0 & 1 & 3 \\
0 & 1 & 1 & 3 \\
1 & 1 & 3 & 0 \\
1 & 1 & 0 & 3
\end{smallmatrix}\right), \left(\begin{smallmatrix}
1 & 0 & 1 & 3 \\
0 & 1 & 3 & 1 \\
1 & 3 & 1 & 0 \\
3 & 1 & 0 & 1
\end{smallmatrix}\right), \left(\begin{smallmatrix}
1 & 0 & 2 & 2 \\
0 & 1 & 2 & 2 \\
1 & 1 & 0 & 3 \\
1 & 1 & 3 & 0
\end{smallmatrix}\right), \left(\begin{smallmatrix}
1 & 0 & 2 & 2 \\
0 & 1 & 2 & 2 \\
1 & 1 & 1 & 2 \\
1 & 1 & 2 & 1
\end{smallmatrix}\right), \left(\begin{smallmatrix}
1 & 0 & 2 & 2 \\
0 & 1 & 2 & 2 \\
1 & 1 & 2 & 1 \\
1 & 1 & 1 & 2
\end{smallmatrix}\right), \left(\begin{smallmatrix}
1 & 0 & 2 & 2 \\
0 & 1 & 2 & 2 \\
1 & 1 & 3 & 0 \\
1 & 1 & 0 & 3
\end{smallmatrix}\right), \left(\begin{smallmatrix}
1 & 0 & 2 & 2 \\
0 & 1 & 2 & 2 \\
2 & 2 & 1 & 0 \\
2 & 2 & 0 & 1
\end{smallmatrix}\right), \left(\begin{smallmatrix}
1 & 1 & 0 & 3 \\
1 & 1 & 3 & 0 \\
0 & 1 & 1 & 3 \\
1 & 0 & 3 & 1
\end{smallmatrix}\right)$, $\left(\begin{smallmatrix}
1 & 1 & 0 & 3 \\
1 & 1 & 3 & 0 \\
0 & 1 & 3 & 1 \\
1 & 0 & 1 & 3
\end{smallmatrix}\right), \left(\begin{smallmatrix}
1 & 1 & 1 & 2 \\
1 & 1 & 2 & 1 \\
1 & 2 & 1 & 1 \\
2 & 1 & 1 & 1
\end{smallmatrix}\right), \left(\begin{smallmatrix}
1 & 1 & 1 & 2 \\
1 & 2 & 1 & 1 \\
1 & 1 & 2 & 1 \\
2 & 1 & 1 & 1
\end{smallmatrix}\right), \left(\begin{smallmatrix}
1 & 1 & 1 & 2 \\
1 & 3 & 0 & 1 \\
1 & 0 & 3 & 1 \\
2 & 1 & 1 & 1
\end{smallmatrix}\right), \left(\begin{smallmatrix}
1 & 2 & 0 & 2 \\
2 & 1 & 0 & 2 \\
0 & 0 & 3 & 2 \\
1 & 1 & 2 & 1
\end{smallmatrix}\right), \left(\begin{smallmatrix}
1 & 2 & 2 & 0 \\
1 & 0 & 1 & 3 \\
1 & 1 & 0 & 3 \\
0 & 2 & 2 & 1
\end{smallmatrix}\right), \left(\begin{smallmatrix}
1 & 2 & 2 & 0 \\
1 & 1 & 0 & 3 \\
1 & 0 & 1 & 3 \\
0 & 2 & 2 & 1
\end{smallmatrix}\right), \left(\begin{smallmatrix}
1 & 4 & 0 & 0 \\
2 & 0 & 0 & 3 \\
0 & 0 & 1 & 4 \\
0 & 1 & 2 & 2
\end{smallmatrix}\right)$, $\left(\begin{smallmatrix}
2 & 0 & 0 & 3 \\
0 & 2 & 3 & 0 \\
0 & 1 & 2 & 2 \\
1 & 0 & 2 & 2
\end{smallmatrix}\right), \left(\begin{smallmatrix}
2 & 0 & 1 & 2 \\
0 & 2 & 2 & 1 \\
1 & 2 & 2 & 0 \\
2 & 1 & 0 & 2
\end{smallmatrix}\right), \left(\begin{smallmatrix}
2 & 0 & 3 & 0 \\
0 & 2 & 0 & 3 \\
1 & 0 & 1 & 3 \\
0 & 1 & 1 & 3
\end{smallmatrix}\right), \left(\begin{smallmatrix}
2 & 1 & 0 & 2 \\
1 & 2 & 0 & 2 \\
0 & 0 & 3 & 2 \\
1 & 1 & 2 & 1
\end{smallmatrix}\right), \left(\begin{smallmatrix}
2 & 1 & 1 & 1 \\
1 & 2 & 1 & 1 \\
1 & 1 & 2 & 1 \\
1 & 1 & 1 & 2
\end{smallmatrix}\right), \left(\begin{smallmatrix}
2 & 1 & 1 & 1 \\
1 & 2 & 1 & 1 \\
1 & 1 & 3 & 0 \\
1 & 1 & 0 & 3
\end{smallmatrix}\right), \left(\begin{smallmatrix}
2 & 3 & 0 & 0 \\
1 & 0 & 0 & 4 \\
0 & 0 & 3 & 2 \\
0 & 2 & 2 & 1
\end{smallmatrix}\right), \left(\begin{smallmatrix}
3 & 0 & 0 & 2 \\
0 & 3 & 0 & 2 \\
0 & 0 & 3 & 2 \\
1 & 1 & 2 & 1
\end{smallmatrix}\right)$, $\left(\begin{smallmatrix}
3 & 0 & 1 & 1 \\
0 & 3 & 1 & 1 \\
1 & 1 & 3 & 0 \\
1 & 1 & 0 & 3
\end{smallmatrix}\right)$

\end{document}